\documentclass[reqno,10pt]{amsart}
\usepackage[all,line,poly,rotate,ps,dvips]{xy}
\SelectTips{cm}{}

\usepackage{comment}
\usepackage{amssymb,eucal,graphicx}
\usepackage{enumerate}

\numberwithin{equation}{section}

\newtheorem{theorem}[equation]{Theorem}

\newtheorem{corollary}[equation]{Corollary}
\newtheorem{claim}[equation]{Claim}
\newtheorem{lemma}[equation]{Lemma}
\newtheorem{proposition}[equation]{Proposition}

\theoremstyle{definition}

\newtheorem{remark}[equation]{Remark}

\theoremstyle{remark}

\newcommand{\C}{\mathbb{C}}

\newcommand{\G}{\mathbb{G}}

\newcommand{\Pp}{\mathbb{P}}

\newcommand{\Ff}{\mathcal{F}}

\newcommand{\HH}{\mathcal{H}}

\newcommand{\N}{\mathcal{N}}

\newcommand{\T}{\mathcal{T}}
\newcommand{\Oc}{\mathcal{O}}
\newcommand{\OO}{\mathcal{O}}

\DeclareMathOperator{\Ext}{{\rm Ext}}

\begin{document}

\title[Brill--Noether loci]
{Brill--Noether loci of stable rank--two vector bundles on a general curve}

\author{C. Ciliberto}
\curraddr{Dipartimento di Matematica, Universit\`a degli Studi di
Roma Tor Vergata\\ Via della Ricerca Scientifica - 00133 Roma
\\Italy} \email{cilibert@mat.uniroma2.it}

\author{F. Flamini}
\curraddr{Dipartimento di Matematica, Universit\`a degli Studi di
Roma Tor Vergata\\ Via della Ricerca Scientifica - 00133 Roma
\\Italy} \email{flamini@mat.uniroma2.it}



\maketitle

\begin{abstract}In this note we give an easy proof of the existence of generically smooth components
of the expected dimension 
of certain Brill--Noether loci of stable rank 2 vector bundles  on a curve with general moduli, 
with related applications to Hilbert scheme of scrolls.
\end{abstract}

\section*{Introduction}\label{sec:intro}

The Brill--Noether theory of  linear series on a smooth, irreducible, complex,  projective curve $C$ of genus $g$ was initiated in the second half of  XIX century and fully developed about one century later by the brilliant work of several mathematicians (see \cite {ACGH} for a general reference).  As  a result, we have now a complete understanding of the \emph{Brill--Noether loci}  of line bundles $L$ of degree $d$ with $h^ 0(C,L)>r$ on a curve $C$ with general moduli. They can be described as determinantal loci  inside $ {\rm Pic}^ d(C)$ and we know their Zariski tangent spaces,  their dimension, their singularities, how they are  contained in each other, etc.

The study of $n$--dimensional scrolls over curves (with $n\ge 2$) also goes back to the second half of XIX century. It is equivalent to 
the study of rank $n$ vector bundles over curves, and as such it has received a lot of attention 
in more recent times. In order to have reasonable moduli spaces for these bundles, one has to restrict the attention to semistable ones.  For them  it has been set up an analogue of   Brill--Noether's theory. Unfortunately the results here are not so complete as in the rank one case, and we are still far from having a full understanding of the situation. We refer the reader to \cite {Tex} (and to the references therein) for a general account on the subject. 

In this paper we deal with the rank $2$ case and $C$ with general moduli and genus $g$. A  result by M. Teixidor (see Theorem \ref {thm:TB}) provides examples of components of the expected dimension (see \eqref {eq:bn} below) of Brill--Noether loci of stable, rank 2 vector bundles $\Ff$ of degree $d$ with $h^ 0(C,\Ff)>\ell$, in suitable ranges for $d, g$ and $\ell$. Teixidor's ingenuous, but not easy, proof uses a degeneration of  $C$  to a rational $g$--cuspidal curve $C_0$ and analyses the limits of the required bundles on $C_0$.

This note is devoted to prove a similar result  (i.e. Theorem \ref {thm:Main}). The ranges for $d, g$ and $\ell$ for which we prove the existence of our components of the  Brill--Noether loci  are slightly worse than Teixidor's ones.  On the other hand we are able to prove a bit more than Teixidor does:  not only  the components in questions have the expected dimension, but  they are also generically smooth. 
In addition, our approach is quite easy and does not require degenerating $C$.   We  construct our bundles  as 
extensions of line bundles, and we prove that their \emph{Petri map} (see \S \ref {ssec:BN}) is in general injective, 
which is the same as proving that the corresponding Brill--Noether loci  are generically  smooth and 
of the expected dimension. 

The  paper is organized  as follows. In \S \ref {sec:prel} we recall the basics about 
moduli spaces of semistable rank--two vector bundles on a curve (see \S \ref {ssec:semist}),
Brill--Noether loci (see \S \ref {ssec:BN}) and Teixidor's theorem (see \S \ref {ssec:teix}). 
The full \S \ref {S:BNnew} is devoted to the construction of our examples. In \S \ref {S:HSL} 
we make some applications to Hilbert schemes of scrolls in projective spaces. We show that our examples  give rise to linearly normal, smooth scrolls belonging to irreducible components of the Hilbert scheme, which are generically smooth of the expected dimension
(see \S\S  \ref {ssec:normal}, \ref {ssecx:comp}). In \S \ref 
{ssec:nonln} we show that, by contrast, their projections  in $\Pp^ {d - 2g
+ 1}$ do not fill up irreducible components of the Hilbert scheme: they are in fact contained in the
unique component $\HH_{d,g}$ of the
Hilbert scheme containing all linearly normal scrolls of degree
$d$ and genus $g$ in $\Pp^{n}$ (cf. \cite[Theorem 1.2]{CCFMLincei} and \cite[Theorem 1]{CCFMnonsp}).\medskip

{\bf Aknowledgements}: It is a pleasure to dedicate this paper to our friend and colleague Gerard van der Geer on the occasion of his 60th birthday.  

We thank A. Verra for useful discussions on the subject of this paper. 

\section{Preliminaries}\label{sec:prel}


\subsection{Moduli spaces of semistable rank--two vector bundles} \label{ssec:semist}
For any integer $d$, we denote by $U_C(d)$ the 
moduli space of rank 2, semistable vector bundles of degree $d$
on $C$. Recall that a   rank 2 vector bundle $\Ff$ of degree $d$ is \emph{semistable} [resp. \emph{stable}]  
if for all quotient line bundles
$\Ff \to\!\!\!\!\! \to L$ of degree $d_1$ one has $d\le 2d_1$ [resp. $d< 2d_1$]. 
$U_C(d)$ is a projective variety and we let  $U^s_C(d)$ be its
open subset whose points correspond to stable vector bundles. 
If $\Ff$ is a semistable rank--two vector bundle on $C$, we 
denote by $[\Ff]$ its class in $U_C(d)$. 

The cases $0\le g\le 1$ are  quite classical and well known (see, e.g., \cite [Chapt. V, \S 2] {Ha},   \cite{Tex, New}).
In general we have (cf.\ \cite[Sect. 5]{New}):

\begin{proposition}\label{prop:moduli} If $g\ge 2$, then:

\begin {itemize}

\item [(i)] $U_C(d)$ is irreducible, normal, of dimension
$4g-3$ and $U_C^s(d)$ is the set of smooth points
of $U_C(d)$;

\item[(ii)] if $d$ is odd, then $U_C(d) = U_C^s(d)$ whereas if $d$
is even, the inclusion $U_C^s(d) \subset U_C(d)$ is strict.
\end{itemize}
\end{proposition}

\subsection{Speciality and Brill--Noether loci}\label {ssec:BN}
If $[\Ff]\in U_C(d)$, we denote by $i(\Ff)$,
or simply by $i$ if there is no danger of confusion, the integer $h^1(C,\Ff)$,
which is called the \emph{speciality} of $\Ff$. Similarly we set
$\ell(\Ff)=h^0(C,\Ff)$, and $r(\Ff)=\ell(\Ff)-1$, and we may often write
$\ell, r$ rather than $\ell(\Ff), r(\Ff)$. By Riemann--Roch theorem, we have
\[
\ell(\Ff)=d-2g+2+i(\Ff).
\]

Fix positive integers $d$ and $i$.  Set $\ell=d-2g+2+i$.
One can consider the subset $B^\ell_C(d)$ of all classes $[\Ff] \in U_C(d)$ such that $i(\Ff)\ge i$
and accordingly $\ell(\Ff)\ge \ell$. This is called the
$\ell^{th}$--{\em Brill-Noether locus} and it has a natural
determinantal scheme structure (see, e.g. \cite{Tex}). A lower bound for the dimension of $B^\ell_C(d)$
as a determinantal locus is its  {\em expected
dimension},  given by the \emph{Brill-Noether number}
\begin{equation}\label{eq:bn}
\rho_d^\ell:= 4g - 3 - i\ell.
\end{equation}

The infinitesimal deformations of $\Ff$ 
along which all sections in $H^0(C,\Ff)$ deform, fill up the 
vector subspace of $H^1(C,\Ff\times \Ff^*)\cong H^0(C, \omega_C \otimes \Ff\times \Ff^*)^*$ which is
the annihilator of the image of the
cup--product map 
\[P_{\Ff} : H^0(C, \Ff) \otimes H^0(C, \omega_C \otimes \Ff^*)
\longrightarrow H^0(C, \omega_C \otimes \Ff \otimes \Ff^*),\]
called the {\em Petri map} of $\Ff$ (see, e.g. \cite{TB00}). 
In other words ${\rm Ann}({\rm Im}(P_{\Ff}))$ is the Zariski tangent space of 
$B^\ell_C(d)$ at $[\Ff]$, where $\deg(\Ff)=d$ and $\ell=h^ 0(C, \Ff)$. In this case, 
by Riemann-Roch theorem,
one has $$\rho_d^\ell = h^1(C, \Ff \otimes \Ff^*)
- h^0(C, \Ff) h^1(C, \Ff).$$
Hence:

\begin{lemma}\label{lem:Petri} In the above setting, $B^\ell_C(d)$ is non--singular, of dimension $\rho_d^\ell$ at $[\Ff]$
if and only if $P_{\Ff}$ is injective.
\end{lemma}

We finish this section by recalling two results.  For the first, see \cite[Proposition 3]{CCFMnonsp}:

\begin{proposition}\label{prop:sstabh1}
Let $C$ be a smooth, irreducible, projective curve of genus $g \ge 2$.  If $d \ge 2g$ then $i(\Ff) = 0$ for $[\Ff] \in U_C(d)$ general.
\end{proposition}

Indeed, we will  be interested in the case $d\ge 2g$ in the rest of this paper. As for the next result, which will somehow justify our construction in \S \ref {S:BNnew}, see 
\cite [Corollary 7.3]{GP1}:

\begin{proposition}\label{prop:FG}
Let $C$ be a smooth, irreducible, projective curve of genus $g \ge 1$ and let $\Ff$ be a special rank 2 vector bundle on $C$. Then there is a quotient $\Ff \to\!\!\!\!\! \to L$ with $L$ a special line bundle.
\end{proposition}

\subsection{A result by M. Texidor} \label{ssec:teix}
If $d \ge 2g$, any rank--two vector bundle $\Ff$ on $C$
has $\ell(\Ff)\ge 2$ by Riemann--Roch theorem. Hence $B^2_C(d) =
U_C(d)$ in this case (cf. \cite[Note, p. 123]{TB0}). 
Then, if  $d\ge 2g$, 
it is no restriction to consider Brill--Noether loci
$B^\ell_C(d)$, with  $\ell \ge 2$. We record here the main result of \cite{TB}:

\begin{theorem}\label{thm:TB} If $\ell \ge 2$, $i \ge 2$, $C$ has 
general moduli, and either $\rho_d^\ell \ge 1$ and $d$ is odd, or $\rho_d^\ell
\ge 5$ and $d$ is even, then
$B^\ell_C(d)$ is not empty and of the expected dimension.
\end{theorem}

\begin{remark}\label{rem:numTB} It is useful to express the numerical 
conditions in Theorem \ref{thm:TB} in terms of the speciality. 
Since $\ell= d - 2g + 2 + i \ge 2$, then $d \ge 2g-i$.
In addition, when $d$ is odd, one has
$\rho_d^\ell \ge 1$, which reads
\begin{equation}\label{eq:dodd}
d \le \frac{i+2}{i}(2g-2) - i;
\end{equation}
when $d$ is even one has $ \rho_d^\ell \ge 5$, i.e.
\begin{equation}\label{eq:deven}
d \le \frac{i+2}{i}(2g-2) - i - \frac{4}{i}.
\end{equation}
\end{remark}

\section{Examples of Brill--Noether loci}\label{S:BNnew}

In this section  we give examples
of generically smooth components of the expected dimension of Brill--Noether loci of speciality $i \geq 1$  in $U_C(d)$,  with $C$ a curve
of genus $g$ with general moduli.

\begin{theorem}\label{thm:Main} Let $g,i$ be a  integers such that 
\begin{equation}\label{eq:ip} i < \sqrt{g+4} - 2.\end{equation}
Let then $d$ and $d_1$ be integers such that 
\begin{equation}\label{eq:bounds1}
g+4 \leq  d_1  \le   (g-i) \frac{(i+1)}{i}
\end{equation} and
\begin{equation}\label{eq:bounds}
 d_1 + g + 3 \leq d < 2d_1 
\end{equation}Set $\ell=d-2g+2+i$.

If $C$ is a curve of genus $g$ with  general moduli,  
there is an irreducible component of $B^\ell_C(d)$
which is generically smooth,  of the expected dimension, containing 
points corresponding to stable, very-ample vector bundles $\Ff$,
with $i(\Ff)=i$, whose minimal degree line bundle quotients  have degree  $d_1$ and speciality $i$. 
\end{theorem}

The proof of Theorem \ref{thm:Main} will
follow from a series of remarks and lemmas presented below. 

\begin{remark}\label {rem:existence} {\rm (i) Note that \eqref {eq:ip} implies  $g\ge 6$ if $i=1$ and $g\ge 13$ if $i\ge 2$. Moreover $i< \frac g 4$ since
$\sqrt{g+4} -2 \le \frac g 4$ for any $g$.
\medskip

\noindent (ii) The interval for the integer $d_1$ in  \eqref{eq:bounds1}  is in general not empty, since 
\begin{equation}\label{eq:ineq}
(g+3) i < (g-i)(i+1) \quad \text{ for any}\quad  i \ge 1.
\end{equation}
If $i=1$, this follows from  $g \geq 6$.  If $i \geq 2$, then \eqref {eq:ineq} 
is equivalent to $i^2 + 4 i - g <0$, which follows from \eqref {eq:ip}. \medskip

\noindent (iii) The inequalities  in \eqref{eq:bounds} are necessary for the stability of $\Ff$ and for the very-ampleness of the line bundle $N$ appearing in \eqref{eq:FUND!} below (cf. Lemmas \ref{lem:num} (i), \ref{lem:step3} and \ref{cl:va}). 
\medskip

\noindent (iv)  The bound
 \[
  d <  2 (g-i) \frac{(i+1)}{i}
  \]
  following from \eqref{eq:bounds1} and  \eqref {eq:bounds}, is in general slightly worse than \eqref {eq:dodd} and \eqref{eq:deven}, but the difference, for $i$ close to the upper bound in   \eqref {eq:ip}, is of the order of $\sqrt g$. \medskip
 
\noindent (v) The upper-bound in 
\eqref{eq:bounds1} implies the following inequality for the {\em Brill--Noether number}
\begin{equation}\label{lem:num} 
\rho(g, d_1, d_1 - g + i)\ge 0.
\end{equation}
}
 \end{remark}\medskip

Now we are going to produce the components of $B_C^\ell(d)$ announced in the statement 
of Theorem \ref {thm:Main}. From \eqref {lem:num}
we have $\dim(W^{d_1-g+i}_{d_1}(C))=\rho(g, d_1, d_1 - g + i) \ge 0$,
because $C$ has general moduli.
Consider  extensions
\begin{equation}\label{eq:FUND!}
0 \to N \to \Ff \to L \to 0,
\end{equation}
with $N \in {\rm Pic}^{d-d_1}(C)$ general and 
$L \in W^{d_1-g+i}_{d_1}(C)$ general 
(or any $L$ if $\rho(g, d_1, d_1- g + i)=0$), so that 
$h^1(C,L)=i$. By  \eqref{eq:bounds}, one has  $d-d_1 \geq g+3$; since $N \in {\rm Pic}^{d-d_1}(C)$ is general, one has $h^1(C,N) = 0$. 
Therefore $\Ff$ is a rank--two vector bundle of degree $d$ and speciality $i=i(\Ff)$, i.e. $[\Ff]\in B_C^\ell(d)$, with $\ell=d-2g+2+i$, and we can look at it as an element of $\Ext^1(L,N)$. 

\begin{lemma}\label{cl:va} In the above setting, any $\Ff \in Ext^1(L,N)$  is very ample on $C$.
\end{lemma}
\begin{proof} A sufficient condition for $\Ff$ to be very-ample is that both $L$ and $N$ are. 
By \cite[(1.8) Theorem, p. 216]{ACGH}, a sufficient condition for 
both $L$ and $N$ to be very ample on $C$ with general moduli is 
\[
h^0(C,N) = d-d_1 - g + 1 \ge 4,\quad h^0(C,L) = d_1 - g + 1 + i \geq 4.
\]
The first inequality holds by \eqref{eq:bounds}, the second by 
 \eqref{eq:bounds1}.  
\end{proof}

\begin{remark}\label{rem:comment} \normalfont{Note that  \eqref {lem:num} and
the proof of Lemma \ref {cl:va} yield
$g-4i=\rho(g, d_1, 3)\ge \rho(g, d_1, d_1- g + i)\ge 0$, so $i\le \frac g 4$ is a necessary condition 
for the ampleness of $L$ (see Remark \ref  {rem:existence}, (i)). } \end{remark}

A general bundle $\Ff \in \Ext^1(L,N)$ as above gives rise to
the projective bundle $\Pp(\Ff)\stackrel{f}{\to} C$, which is embedded, via 
 $\vert \OO_{\Pp(\Ff)}(1)\vert$,
as a smooth scroll $S$ of degree $d$ and sectional genus $g$
in $\Pp^r$,  with $r=r(\Ff) = d-2g+1+i$. The quotient 
$\Ff \to\!\!\!\!\! \to L$ corresponds to a section $C\stackrel{s}{\to} \Pp(\Ff)$
of $\Pp(\Ff)\stackrel{f}{\to} C$, whose image is a unisecant, irreducible curve $\Gamma\cong C$
(cf.\ \cite[\S\;V, Prop.\ 2.6 and 2.9]{Ha}). 
Since $h^ 1(C, N)=0$, then $\Gamma\subset S\subset \Pp^r$ is linearly normally embedded 
in a linear subspace of  dimension $d_1-g+i$, as a curve of 
degree $d_1$ and speciality $i$.

Given $L \in W^{d_1-g+i}_{d_1}(C)$,  $N \in {\rm Pic}^{d-d_1}(C)$ and $\Ff$  in
$\Pp(\Ext^1(L,N))$, the embedding $\Pp(\Ff)\to \Pp^r$ varies by projective
automorphisms of $\Pp^r$. Thus the surface $S$ varies, describing an irreducible
locally closed subset $\mathcal H_C(d,i)$ of the Hilbert scheme. 

\begin{remark}\label{rem:hilbertscheme} \normalfont{It is 
useful to describe $\mathcal H_C(d,i)$ in a different way. 

Let $L \in W^{d_1-g+i}_{d_1}(C)$ be general as above.
Let $M \in {\rm Pic}^{\delta}(C)$ be any line bundle of
degree $\delta >>0$. Consider the projective bundle 
$\Pp(L \oplus M)$, which embeds as a smooth scroll 
$\Sigma=\Sigma_{L,M}$ of degree $d_1 + \delta$ and sectional genus
$g$ in $\Pp^{r'}$, $r'=r(L \oplus M) = d_1 + \delta - 2g +
1 + i = r + \delta$, via  $\vert \OO_{\Pp(L \oplus M)}(1)\vert$.
By \cite[Thm. 3.11]{CCFMsp}, $\Sigma$ contains a unique special section $E$ 
of degree $d_1$ and speciality  $i$, corresponding to the quotient 
$L \oplus M \to\!\!\!\!\! \to L$. One has 
$E^2 = d_1 - \delta <<0$ (cf.\ \cite[\S\,5]{Ha}). 

One can choose $M$ in such a way that there are 
$\delta + d_1 - d$ linearly
independent points on $\Sigma$ such that, projecting
$\Sigma$ from their span, the
image of this projection is the surface $S \subset \Pp^r$
as above. In this projection $E$ is isomorphically
mapped to $\Gamma$.

Thus, $\mathcal H_C(d,i)$ can be  thought of 
as the family of all projections of scrolls of the form $\Sigma_{L,M}$,
with $L, M$ as above, from
the span of $\delta + d_1 - d$ general points on $\Sigma_{L,M}$. 
}\end{remark}

\begin{lemma} \label{lem:step3} In the above setting, $\Gamma$
is the unique section of minimal degree of the scroll $S$. Hence
the bundle $\Ff$ is stable.
\end{lemma}

\begin{proof} Assume, by contradiction,  there
is a section $\Gamma' \subset S$  of degree $d_1 - \nu$, with
$\nu\ge 0$.

Let $\Sigma=\Sigma_{L,M}$ be a scroll as in Remark
\ref {rem:hilbertscheme} and let   
$\varphi : \Sigma \dasharrow S$ be the projection of $\Sigma$ to $S$
from suitable $\delta + d_1 - d$ general points on it. Let 
$X$ be the set of these points.
Then $\Gamma'$ is the image via $\varphi$ of a section
$\Delta \neq E$ of degree $d_1 - \nu + h$ passing through a subset $Y\subseteq X$
of  $h$  points, for some $h\ge 0$. If we denote by
$F$ the ruling of $\Sigma$, then $\Delta \equiv E + (h-\nu)
F$. Then:
\begin{itemize}
\item[(i)] $\Delta\cdot  E \ge 0$ implies $E^2 + h - \nu \ge 0$,  therefore  $\Delta^2 = E^2 + 2 (h-\nu) \ge -E^ 2>>0$;

 \item[(ii)] hence $h^1(\Delta, \N_{\Delta\vert  \Sigma}) = 0 $ so $h^0(\Delta, \N_{\Delta\vert  \Sigma}) =
d_1 - \delta + 2 (h - \nu) - g + 1$. Since $Y$ consists of  $h$ general points, 
we must have $d_1 - \delta + 2 (h - \nu) - g + 1 \ge h$,
i.e.  $h - 2 \nu \ge \delta - d_1 + g - 1$.
\end{itemize}

Putting  the above  inequalities together, we have
\[
\delta + d_1 - d \ge h \ge h - 2 \nu \ge \delta - d_1 + g - 1 \quad \text {hence}\quad d \leq 2d_1 - g + 1
\]
which, by the first inequality in \eqref{eq:bounds},  implies $d_1 \ge 2g+2$,  contrary to the fact that $L$ is special. 
This proves the first assertion. Then the stability of $\Ff$  
follows from $d<2d_1$ in \eqref {eq:bounds}.
\end{proof}

\begin{remark} \label{rem:dimension} \normalfont{Let us 
compute
$y:=\dim(\mathcal H_C(d,i))-\dim({\rm PGL}(r+1,\C))$. 
A scroll $S$ corresponding to a point of 
$\mathcal H_C(d,i)$ is of the type $\Pp(\Ff)$, with $\Ff$ 
an extension as in \eqref {eq:FUND!}. 
By Lemma \ref {lem:step3}, this extension is \emph{essentially unique},
i.e. two of them correspond to the same point of $\Pp(\Ext^1(L,N))$ (cf.
\cite[p. 31]{Fr}). Therefore $y$ is the sum of the following  quantities:

\begin{itemize}
\item  $\rho(g, d_1, d_1 - g + i)$, i.e. the number of
parameters for the line bundle $L$;

\item $g$, i.e. the number of parameters for the line bundle $N$; 

\item $\dim(\Pp(\Ext^1(L,N)))=2 d_1 - d + g -
2$: indeed,  $\deg(N-L) = d - 2 d_1
<0$, thus $h^1(C, N\otimes L^{*}) = 2 d_1 - d + g -
1$.
\end{itemize} 

Consider the 
\emph{modular map} $\mu: \mathcal H_C(d,i)\dasharrow B^\ell_d(C)$
sending the point corresponding to $S\cong \Pp(\Ff)$ to $[\Ff]$.
This is well defined  since $S\cong \Pp(\Ff)\cong \Pp(\Ff')$ implies
$\Ff\cong\Ff'$. The fibres of
$\mu$ are  orbits by the ${\rm PGL}(r+1,\C)$--action on $\mathcal H_C(d,i)$. Therefore $y$
is the dimension of the image of $\mu$, hence $\dim(B^\ell_d(C))\ge y$. 
}\end{remark}

The next lemma shows that
the image of $\mu$ lies in a component of $B^\ell_d(C)$ which is generically smooth and of the expected dimension, thus 
concluding the proof of Theorem \ref {thm:Main}.

\begin {lemma}\label{lem:conclusion} Let $\Ff$ be a bundle appearing 
in \eqref{eq:FUND!} with $L \in W^{d_1-g+i}_{d_1}(C)$
and $N \in {\rm Pic}^{d-d_1}(C)$ general. Then the Petri map $P_\Ff$ is injective.
\end{lemma} 

\begin{proof} For all $\Ff\in \Ext^1(L,N)$, one has $h^1(C, \Ff)=i$, hence the domain of
$P_\Ff$ has constant dimension $i(d-2g+2+i)$. Therefore, by semicontinuity, it suffices to prove the assertion
for a particular such $\Ff$, even if the dimension of the target of $P_\Ff$ jumps up. We
will specialize to  $\Ff_0 = L \oplus N$. We have
$$H^0(C, \Ff_0) = H^0(C, L)\oplus H^0(C, N) \;\; {\rm and} \;\; H^0(C, \omega_C \otimes \Ff_0^{*}) = H^0(C, \omega_C \otimes L^{*}).$$
So the domain of $P_{\Ff_0}$ is 
$$H^0(C, \Ff_0) \otimes H^0(C, \omega_C \otimes \Ff_0^{*}) = \left(H^0(C, L) \otimes H^0(C, \omega_C \otimes L^{*})\right)
\oplus \left(H^0(C, N) \otimes H^0(C, \omega_C \otimes
L^{*})\right),$$
whereas the target is
$$H^0(C, \omega_C \otimes \Ff_0 \otimes \Ff_0^{*}) = H^0(C, \omega_C) \oplus H^0(C, \omega_C \otimes L \otimes N^{*})
\oplus H^0(C, \omega_C \otimes N \otimes L^{*}) \oplus
H^0(C, \omega_C).$$
The map $P_{\Ff_0}$ can be
written on decomposable tensors as
$$(a \otimes b, \alpha \otimes \beta) \stackrel{P_{\Ff_0}}{\longrightarrow} (ab, 0 , \alpha \beta, 0),$$
for $a \otimes b \in H^0(C, L) \otimes H^0(C, \omega_C \otimes L^{*})$
and for $\alpha \otimes \beta \in H^0(C, N) \otimes H^0(C, \omega_C
\otimes L^{*})$. In other words,
$$P_{\Ff_0} = \mu_L \oplus \mu_{L,N}$$
where $\mu_L : H^0(C, L) \otimes H^0(C, \omega_C \otimes L^{*}) \to H^0(C, \omega_C)$ is the Petri map for $L$ and
$\mu_{L,N} : H^0(C, N) \otimes H^0(C, \omega_C \otimes L^{*}) \to
H^0(C, \omega_C \otimes N \otimes L^{*})$ is the multiplication
map.

Since $C$ has general moduli, the map $\mu_L$ is injective. We need to prove that $\mu_{L,N}$ is also injective. To do this, it suffices to show that $\mu_{L,N}$ is injective for some
particular line bundle $N_0$ of degree $d-d_1$, even if $N_0$ becomes special and therefore $h^ 0(C,N_0)>h^ 0(C,N)=d-d_1-g+1$. Indeed, when a general $N$ flatly tends to $N_0$, the vector spaces $H^ 0(C,N)$  and $H^0(C, \omega_C \otimes N \otimes L^{*})$
will  respectively tend to subspaces $V\subseteq H^ 0(C,N_0)$ and $W\subseteq
H^0(C, \omega_C \otimes N_0 \otimes L^{*})$ of the same dimensions,  
and the limit of $\mu_{L,N}$ will be the multiplication map
$\mu_{L,V} : V \otimes H^0(C, \omega_C \otimes L^{*}) \to W$. Hence 
$\mu_{L,V}$ (hence, by semicontinuity, $\mu_{L,N}$) is injective if $\mu_{L,N_0}$ is.

Let $\Delta \in
{\rm Div}^{2d_1-d}(C)$ be an effective divisor.
Let $N_0 = L(-\Delta) \in {\rm Pic}^{d-d_1}(C)$ and set $\mu_0=\mu_{L,N_0}$.
If we tensor the exact sequence
$$0 \to L(-\Delta)\cong N_0 \to L \to L|_{\Delta} \to 0$$
by $H^0(C,\omega_C \otimes L^{*})$, we get
the  commutative diagram with exact rows

\[
\begin{array}{rccccc}
0 \to & H^0(C,N_0) \otimes H^0(C, \omega_C \otimes L^{*}) &
\longrightarrow & H^0(C, L) \otimes H^0(C, \omega_C \otimes L^{*})\\
 & \downarrow^{\mu_{0}} & & \downarrow^{\mu_L} \\
 0 \to & H^0(C, \omega_C (-\Delta)) & \longrightarrow & H^0(C, \omega_C ) 
\end{array}
\]
Since $\mu_L$ is injective, $\mu_0$ is also
injective, which ends our proof. \end{proof}

\begin{remark}\label{rem:impo} Except in the case $i=1$ and $d_1 = 2g-2$, the general point of a component of $B^\ell_d(C)$ we constructed  does not lie in the image of  $\mu$. Indeed, by recalling \eqref {eq:bn}, one has
$$y - \rho^\ell_d = d(i-1) - d_1(i-2) - (g -1) (i+1).$$\medskip

\noindent (i) When $i=1$, one has $y- \rho^{d-2g+3}_d = d_1 - 2 (g-1)$, which is zero if and only if $d_1 = 2g-2$.  If  $i=1$ consider a general vector bundle $\Ff$ in our component. By
Proposition \ref {prop:FG}, there is an exact sequence of the form \eqref {eq:FUND!} with $h^ 1(C,L)>0$, hence $h^ 1(C,L)=1$. Then the above argument shows that $L\cong \omega_C$. \medskip

\noindent (iii) When $i \geq 2$, by $d < 2d_1$ and \eqref {eq:bounds1} one has $$ y - \rho^\ell_d < i (d_1 - g + 1) + 1 - g \leq 1 - i^2 < 0.$$ 

The problem of describing the general element of a component of $B^\ell_d(C)$ we constructed when $i>1$ (the case $i=1$ is treated in (i)) looks interesting and we plan to come back to it in a future research. 
\end{remark}

\begin{corollary}\label{prop:Main} Let  $C$ a general curve of genus $g\ge 6$. 
For any $3g+1 \le d \le 4g-5$, there is a component of $B^{d-2g+3}_d(C)$, which is generically smooth and of the expected dimension, whose general point  corresponds to a very ample, stable vector bundle $\Ff$ of speciality $1$, fitting in an exact sequence $$0 \to N \to \Ff \to \omega_C \to 0$$where $\omega_C$ is the minimal degree quotient line bundle of $\Ff$ and $N \in Pic^{d-2g+2}(C)$ is general. 
\end{corollary}

\section{Applications to Hilbert schemes of scrolls}\label{S:HSL}

In this section we use Theorem \ref{thm:Main} to study
some components of  Hilbert schemes of special scrolls. 

\subsection{Normal bundle cohomology} \label{ssec:normal} Here we prove the following:

\begin{proposition}\label{prop:tghilbi}
Assumptions as in Theorem \ref{thm:Main}. Let $r = d-2g+1+i$ and 
$S \subset \Pp^r$ be a smooth, linearly normal, special scroll of
degree $d$, genus $g$, speciality $i$, with general moduli, which corresponds to a general 
point of  $\mathcal H_C(d,i)$ as in Remark \ref{rem:hilbertscheme}. 
 If $\N_{S\vert  \Pp^{r}}$ is the normal bundle of $S$ in
$\Pp^{r}$, then:
\begin{itemize}
\item[(i)] $h^0( S, \N_{S\vert  \Pp^{r}}) = 7(g-1) + (r +1) (r +1 - i)$;
\item[(ii)] $h^1( S, \N_{S\vert  \Pp^{r}}) = 0 $;
 \item[(iii)] $h^2( S,
\N_{S\vert  \Pp^{r}}) = 0$.
\end{itemize}
\end{proposition}

\begin{proof}[Proof of Proposition \ref{prop:tghilbi}]  First, we prove (iii).  Since $S$ is linearly normal, from 
{\em Euler's  sequence} we get: $$ \cdots \to H^0(S, \Oc_S(H))^* \otimes
H^2(S, \Oc_S(H)) \to H^2 (S, T_{\Pp^r}|_S) \to 0$$
where $H$ is  a hyperplane section of $S$. 
Since $S$ is a
scroll, then $h^2(S, \Oc_S(H)) = 0$, which implies $h^2
(S, \T_{\Pp^r}|_S)= 0$. Thus (iii) follows by the {\em normal bundle sequence}
\begin{equation}\label{eq:tang}
0 \to \T_S \to \T_{\Pp^{r}}|_S \to \N_{S\vert  \Pp^{r}} \to 0.
\end{equation}

Since $S$ is a scroll of genus $g$, we have
\begin{equation}\label{eq:tgS}
\chi(\mathcal O_S) = 1-g, \quad  \chi(\T_S) = 6 - 6g.
\end{equation}
Since $S$ is linearly
normal, from  Euler's sequence we then get
\begin{equation}\label{eq:tgS2}
\chi(\T_{\Pp^{r}}|_S) = (r +1) (r +1 - i) + g-1.
\end{equation}Thus, from (iii) and
 \eqref{eq:tgS}, \eqref{eq:tgS2} we get
\begin{equation}\label{eq:tgS3bis}
\chi(\N_{S\vert  \Pp^{r}}) = h^0(S, \N_{S\vert  \Pp^{R}}) - h^1(S, \N_{S\vert  \Pp^{r}})=
7(g-1) + (r +1) (r +1 - i).
\end{equation}

The rest of the proof is concentrated on  computing
$h^1(S, \N_{S\vert \Pp^{r}})$.

Since $S=\Pp(\Ff)$ is a scroll corresponding to a general point $[\Ff]\in \mathcal H_C(d,i)$, let $\Gamma$ be the unisecant of $S$ of degree $d_1$ corresponding to the special quotient line bundle $\Ff \to\!\!\!\!\! \to L$.

\begin{claim}\label{cl:ciro1502} One has $h^1(S, \N_{S\vert \Pp^{r}} (-\Gamma)) = h^2(S, \N_{S\vert \Pp^{r}} (-\Gamma))  = 0$, hence
\begin{equation}\label{eq:tgS15}
h^1(S, \N_{S\vert \Pp^{r}}) = h^1(\Gamma, \N_{S\vert \Pp^{r}}|_{\Gamma}).
\end{equation}
\end{claim}

\begin{proof}[Proof of Claim \ref{cl:ciro1502}] Look at the exact
sequence
\[0 \to \N_{S\vert \Pp^{r}} (-\Gamma) \to  \N_{S\vert \Pp^{r}} \to
\N_{S\vert \Pp^{r}}|_{\Gamma} \to 0.\]
From \eqref{eq:tang} tensored by $\Oc_S(-\Gamma)$ we see that $h^2(S, N_{S\vert \Pp^{r}} (-\Gamma)) = 0$
follows from $h^2(S,  \T_{\Pp^r}|_S (-\Gamma)) = 0$ which, by 
Euler's sequence, follows from $h^2(S, \Oc_S(H - \Gamma)) =
h^0(S, \Oc_S(K_S - H + \Gamma)) = 0$, since $K_S - H + \Gamma$
intersects the ruling of $S$ negatively.

As for $h^1(S, \N_{S\vert \Pp^{r}} (-\Gamma))= 0$, this follows from $h^1(S,
\T_{\Pp^r}|_S (-\Gamma)) = h^2(S,  \T_S (-\Gamma)) = 0$. By  Euler's
sequence, the first vanishing follows from $h^2(S, \Oc_S(-\Gamma)) =
h^1(S, \Oc_S(H-\Gamma))=0$. Since $K_S+\Gamma$ meets the ruling
negatively, one has $h^0(S, \Oc_S(K_S +\Gamma)) = h^2(S, \Oc_S(-\Gamma))
=0$. Moreover
$h^1(S, \Oc_S(H-\Gamma)) = h^1(C, N)=0$.

In order  to prove $h^2(S,  \T_S (-\Gamma)) = 0$, consider the exact
sequence
\[0 \to \T_{rel} \to \T_S \to \rho^*(\T_C) \to 0\]
arising from the structure morphism
$S= \Pp(\Ff) \stackrel{\rho}{\to} C$. The vanishing we
need follows from $h^2(S,  \T_{rel} \otimes \Oc_S(-\Gamma)) = h^2
(S, \Oc_S(-\Gamma) \otimes \rho^*(\T_C)) = 0$. The first vanishing
holds since $\T_{rel} \cong \Oc_S (2H - dF)$, where $F$ is a ruling of $S$,
 and therefore,
$\Oc_S(K_S + \Gamma) \otimes \T_{rel}^{*}$ restricts negatively
to the ruling. Similar considerations  yield the second
vanishing. \end{proof}

Consider the exact sequence
\begin{equation}\label{eq:B}
0 \to  \N_{\Gamma\vert  S}  \to   \N_{\Gamma\vert \Pp^{r}}  \to
\N_{S\vert \Pp^{r}}|_{\Gamma}  \to 0.
\end{equation}

\begin{claim}\label{cl:flam2611} The map $H^1(\Gamma, \N_{\Gamma\vert  S}) \stackrel{\alpha}{\longrightarrow}
H^1(\Gamma, \N_{\Gamma\vert \Pp^{r}})$ arising from \eqref{eq:B} is
surjective, hence
$h^1(\Gamma, \N_{S\vert \Pp^{r}}|_{\Gamma})= 0$.
\end{claim}

\begin{proof}[Proof of Claim \ref{cl:flam2611}] Equivalently, we show the injectivity of the dual map
\begin{equation}\label{eq:dual}
H^0(\Gamma, \omega_{\Gamma} \otimes \N_{\Gamma, \Gamma\vert \Pp^{R}}^{*})
\stackrel{\alpha^{*}}{\longrightarrow} H^0(\Gamma, \omega_{\Gamma}
\otimes \N_{\Gamma\vert  S}^{*}) \cong H^0(C, \omega_C \otimes N \otimes
L^{*}).\end{equation}
Consider $\Gamma \subset \Pp^h$,
where $h = d_1 - g + i$,  and the Euler sequence of $\Pp^h$
restricted to $\Gamma$. By taking cohomology and  dualizing, we
get $$0 \to H^1(\Gamma,  \T_{\Pp^h}|_{\Gamma})^{*} \to H^0
(\Gamma, \Oc_{\Gamma}(H)) \otimes H^0(\Gamma, \omega_{\Gamma} (-H))
\stackrel{\mu_0}{\to} H^0(\Gamma, \omega_{\Gamma}),$$where $\mu_0$ is the
 Brill-Noether map of $\Oc_{\Gamma}(H)$. Since $\Gamma \cong
C$ has general moduli, then $\mu_0$ is injective by
Gieseker-Petri's theorem  (cf. \cite{ACGH}) so $h^1(\Gamma,  \T_{\Pp^h}|_{\Gamma})=
0$. From the exact sequence
$$0 \to \T_{\Gamma} \to  \T_{\Pp^h}|_{\Gamma} \to \N_{\Gamma\vert \Pp^h} \to 0$$ we get
$h^1(\Gamma,  \N_{\Gamma\vert \Pp^h})= 0$. From the inclusions
$\Gamma \subset \Pp^h \subset \Pp^R$ we have the sequence
\[ 0 \to \N_{\Gamma\vert \Pp^{h}} \to \N_{\Gamma\vert \Pp^{R}} \to
\N_{\Pp^h\vert \Pp^{R}}|_{\Gamma} \to 0,\]
which shows that $H^1(\Gamma, \N_{\Gamma\vert \Pp^{R}}) \cong H^1(\Gamma, \N_{\Pp^h\vert \Pp^{R}}|_{\Gamma})$, i.e.
\begin{equation}\label{eq:ganzissimo}
H^0(\Gamma, \omega_{\Gamma} \otimes \N_{\Gamma\vert \Pp^{R}}^{*}) \cong
H^0(\Gamma, \omega_{\Gamma} \otimes \N_{\Pp^h\vert \Pp^{R}}|_{\Gamma}^{*}).
\end{equation} On the other hand,
from \eqref{eq:FUND!} and the non-speciality of $N$, we get $$0
\to H^0(C, L)^{*} \to H^0(C, \Ff)^{*} \to H^0(C, N)^{*} \to
0.$$Since $H^0(S, \Oc_S(H)) \cong H^0(C, \Ff)$ and $\Oc_{\Gamma}(H)
\cong L$, the Euler sequences restricted to $\Gamma$ give the
following commutative diagram
\begin{displaymath}
\begin{array}{ccccccc}
      &             &     &     0                       &     &     0     &     \\
       &             &     &     \downarrow                        &     &     \downarrow      &     \\
0 \to & \Oc_{\Gamma} & \to & H^0(C, L)^{*} \otimes L
& \to & \T_{\Pp^h}|_{\Gamma} & \to 0  \\
      &      ||       &     &     \downarrow                        &     &     \downarrow      &     \\
0 \to & \Oc_{\Gamma} & \to & H^0(C, \Ff)^{*} \otimes
L & \to & \T_{\Pp^r}|_{\Gamma} & \to 0  \\
     &             &     &     \downarrow                        &     &     \downarrow      &     \\
 &  &  & H^0(C, N)^{*} \otimes L&\cong  & \N_{\Pp^h\vert \Pp^{r}}|_{\Gamma} & \\
          &             &     &     \downarrow                        &     &     \downarrow      &     \\
     &             &     &     0                       &     &     0     &

\end{array}
\end{displaymath}
This gives
\begin{equation}\label{eq:ganzissimo2}
H^0(\Gamma, \omega_{\Gamma} \otimes \N_{\Pp^h\vert \Pp^{R}}|_{\Gamma}^{*})
\cong H^0(C, N) \otimes H^0(C, \omega \otimes L^{*}).
\end{equation}By \eqref{eq:dual}, \eqref{eq:ganzissimo} and \eqref{eq:ganzissimo2}, we  see that 
$\alpha^{*} = \mu_{L,N}$,
whose injectivity has been shown in Lemma \ref{lem:conclusion}.
\end{proof}

From Claim \ref {cl:flam2611},  \eqref{eq:tgS3bis} and \eqref{eq:tgS15}, both (i) and (ii)
follow.
\end{proof}

\subsection{Components of the Hilbert scheme of linearly normal, special scrolls}\label {ssecx:comp}
We  denote by $ {\rm Hilb} (d,g,i)$ the open
subset of the Hilbert scheme parametrizing smooth scrolls in
$\Pp^r$ of genus $g$, degree $d$ and speciality
$i$, with $r = d-2g+1+i$.

\begin{theorem}\label{thm:Main2} Numerical assumptions as in Theorem \ref{thm:Main}. Then
$ {\rm Hilb} (d,g,i)$ has an irreducible component  $\HH$ which contains all
$\HH_C(d,i)$ with $C$ a general curve of genus $g$.
The general point
$[S] \in \HH$ is a smooth scroll of degree $d$, genus $g$ and
speciality $i$, which is linearly normal in $\Pp^r$. Moreover:
\begin{itemize}
\item[(i)] $\HH$ is generically smooth  of dimension $\dim(\HH)
= 7g-7 + (r+1) (r+1-i)$; 
\item[(ii)] $[S] \in \HH$ general corresponds to a pair $(\Ff,C)$, where $C$ has general moduli and
$\Ff$ is stable  of speciality $i$ on $C$.
\end{itemize}  When $i =1$ and $ 3g+1 \le d \le 4g-5$,  the union of all
$\HH_C(d,i)$ with $C$ a general curve of genus $g$ is dense in 
$\HH$ and the general scroll $[S] \in \HH$ has a canonical curve as the unique special section of minimal degree. 
\end{theorem}

\begin{proof} The construction of $\HH$ is clear. Its generic smoothness and 
the dimension count
follow from Proposition \ref{prop:tghilbi}.
The last part of the statement  follows from Corollary \ref{prop:Main}.  \end{proof}

\begin{remark}\label{rem:vaje2} As we saw in Remark \ref {rem:impo}, the union of 
 $\HH_C(d,i)$ with $C$ a general curve of genus $g$ is never dense in $\HH$ unless $i=1$ and $d_1=2g-2$. \end{remark}

\begin{remark}\label{rem:unstable} In \cite{CCFMsp} we constructed components of Hilbert schemes parametrizing smooth, linearly normal, special scrolls $S \subset \Pp^r$, of degree $d$, genus $g$ having the base curve with general moduli. Such components were constructed for any $g \geq 3$, $i \ge 1$ and for any 
$d \geq \frac{7g-\epsilon}{2} - 2 i + 2$, $0 \le \epsilon \le 1$, $\epsilon \equiv g \pmod 2$, unless $i =2$ where 
$d \ge 4g-3$ (cf. \cite[Thm. 6.1]{CCFMsp}). The general point of any such component corresponds to an unstable vector bundle  on $C$ (cf. \cite[Rem. 6.3]{CCFMsp}). 
\end{remark}

\subsection{Non-linearly normal, special scrolls}\label{ssec:nonln}

Let $n = d - 2g
+ 1 $. Recall that there is a unique component $\HH_{d,g}$ of the
Hilbert scheme containing all linearly normal, non--special scrolls of degree
$d$ and genus $g$ in $\Pp^{n}$ (cf. \cite[Theorem 1.2]{CCFMLincei} and \cite[Theorem 1]{CCFMnonsp}).

Consider now  the family $\mathcal Y_{i}$ whose general
element is a general projection to $\Pp^n$ of the scroll $S
\subset \Pp^r$, $r = n+i$, with $[S] \in  \HH$ general as in
Theorem \ref{thm:Main2}. 
The following proposition shows that the families  $\mathcal Y_{i}$ never fill up components of the Hilbert scheme of $\Pp^ n$.

\begin{proposition}\label{prop:ciropasqua} In the above setting, for  $d$, $d_1$, $g$ and $i$ as in Theorem \ref{thm:Main}, $\mathcal Y_{i}$ is a generically smooth subset of $\HH_{d,g}$ of codimension $i^2$ whose general point is smooth for  $\HH_{d,g}$.
\end{proposition}

\begin{proof} 
Let $[S] \in \HH$ be
general  with $S \cong \Pp(\Ff)$ and let  $S' \subset \Pp^n$ be a general
projection of $S$. Let $G_{S'} \subset {\rm
PGL}(n+1, \C)$ be the subgroup of projective transformations  fixing $S'$.
Since $G_S \subset {\rm Aut}(S) \cong {\rm Aut}(\Pp(\Ff)) $, one has $\dim(G_{S'})= 0$, because  $\Ff$ is stable.

Then  $\dim(\mathcal Y_i)$ is:
\begin{itemize}
\item $3g -3$, for the parameters on which $C$ depends, plus \item
$4g-3 -  i (r+1)$, for the parameters on which $\Ff$
depends, plus \item $\dim (\G (n, r)) = (n+1) (r-n) = (n+1) i $,
which are the parameters for the projections, plus \item $(n +1)^2
-1 = \dim ({\rm PGL}(n+1, \C))$, minus \item $\dim(G_{S'})=0$.
\end{itemize} 
Adding up, we get $\dim(\mathcal Y_i) = \dim (\HH_{d,g}) - i^2$. 

Consider the {\em Rohn exact sequence}
$$0 \to \C^i \otimes \Oc_S(H) \to \N_{S\vert \Pp^{r}} \to \N_{S'\vert \Pp^{n}} \to 0$$(see,
e.g. \cite{Cil}, p. 358, formula (2.2)). From Proposition
\ref{prop:tghilbi}, (ii), we have $h^1(S, \N_{S\vert \Pp^r}) =0$, therefore
also $h^1(S',\N_{S'\vert \Pp^{r}}) = 0$. Hence $\mathcal Y_i$ is contained
in a component $\mathcal Z$ of the Hilbert scheme of dimension
$\chi(\N_{S'\vert \Pp^{r}} ) = 7(g-1) + (r+1)^2$ and the general point
of $\mathcal Y_i$ is a smooth point of $\mathcal Z$.

The general point of $\mathcal Y_i$ is a smooth
scroll on $C$ arising from a stable, rank-two vector bundle. The
component $\HH_{d,g}$ is the only
component of the Hilbert scheme whose general point corresponds to a stable scroll (cf. the proof of \cite[Theorem 2]{CCFMnonsp}).
Therefore, $\mathcal Z = \HH_{d,g}$. The map
$H^0(S,\N_{S\vert \Pp^{r+1}}) \to H^0(S',\N_{S'\vert \Pp^{r}}) $ is not
surjective: its cokernel is $\C^i \otimes H^1(\Oc_S(H))^{\oplus i}
$, which has dimension $i^2$. This means that $\mathcal Y_i$ is a
generically smooth subset of $\HH_{d,g}$ of codimension $i^2$.
\end{proof}


\end{document}